\documentclass[11pt]{amsart}
\usepackage{amsfonts}
\usepackage{amsmath}
\usepackage{amsthm, lipsum}

\swapnumbers
\newtheorem{thm}{Theorem}[section]

\newtheorem{cor}[thm]{Corollary}

\newtheorem{lemma}[thm]{Lemma}
\newtheorem{prop}[thm]{Proposition}
\theoremstyle{definition}

\theoremstyle{remark}

\newtheorem{ex}[thm]{Example}

\newtheorem{question}[thm]{Question}

\DeclareMathOperator{\TR}{tr}
\DeclareMathOperator{\SPAN}{span}

\newcommand\UNFs[1]{\Omega_{N,M} {(\mathbb {#1})} } 
\newcommand\Gnk[1]{\mu_{N,M} (\mathbb #1)} 
\newcommand\textem[1]{{\em #1}}

\usepackage{lineno,hyperref}
\modulolinenumbers[5]

\begin{document}

\title{On the rigidity of geometric and spectral properties of Grassmannian frames}


\author{Peter G. Casazza}
\email{Casazzap@missouri.edu }
\author{John I. Haas}
\email{haasji@missouri.edu}

\address{Department of Mathematics, University
of Missouri, Columbia, MO 65211-4100}

\begin{abstract}
We study the rigidity properties of Grassmannian frames: basis-like sets of unit vectors that correspond to optimal Grassmannian line packings.  It is known that Grassmannian frames characterized by the Welch bound must satisfy the restrictive geometric and spectral conditions of being both equiangular and tight; however, less is known about the neseccary properties of other types of Grassmannian frames.
We examine explicit low-dimensional examples of orthoplectic Grassmannian frames and conclude that, in general, the necessary conditions for the existence of Grassmannian frames can be much less restrictive.  In particular, we exhibit a pair of $5$-element Grassmannian frames in $\mathbb C^2$ manifesting with differently sized angle sets and different reconstructive properties (ie, only one of them is a tight frame).  This illustrates the complexity of the line packing problem, as there are cases where a solution may coexist with another solution of a different geometric and spectral character.  Nevertheless, we find that these "twin" instances still respect a certain rigidity, as there is a necessary trade-off between their tightness properties and the cardinalities of their angle sets.  The proof of this depends  on the observation that the traceless embedding of Conway, Hardin and Sloane sends the vectors of a unit-norm, tight frame to a zero-summing set on a higher dimensional sphere.  In addition, we review some of the known bounds for characterizing optimal line packings in $\mathbb C^2$ and discuss several examples of Grassmannian frames achieving them.
\end{abstract}


\maketitle

\linenumbers

\section{Introduction}

Due to their usefulness in numerous areas of science~\cite{MR2662471,MR2059685, MR2778089}, engineering~\cite{5946867,MR2028016,MR2921716, MR2021601} and mathematics ~\cite{MR2149656, fickus2016tremain, MR1984549}, \textem{equiangular tight frames (ETFs)} are a class of frames that have received much attention in recent years ~\cite{MR3150919, MR2277977, MR2015832, MR2736147, MR2350682, MR3005267, Tremain_08, MR2235693, Zau, Fickus:2015aa};  however, ETFs rarely exist~\cite{szollHosi2014all}.  Because a compactness argument shows that Grassmannian frames always exist ~\cite{MR1984549},  they  often serve as ideal generalizations for ETFs in applications where low coherence is desired ~\cite{bgb15, MR1984549, bodmann2015achieving, MR2021601, MR3429342}.   

A complex \textem{Grassmannian frame} is a set of $N$ unit vectors that spans $\mathbb C^M$ with the property that the maximal element of its \textem{angle set}, or set of  pairwise absolute inner products, is minimal. A Grassmannian frame is \textem{$K$-angular} if the cardinality of its angle set is $K$  and it is \textem{tight} if it satisfies a scaled version of Parseval's identity.

Equiangular tight frames are $1$-angular Grassmannian frames characterized by achievement of the optimal lower bound of Welch~\cite{MR2059685, MR2711357}, but this is only possible if $N \leq M^2$ ~\cite{MR1984549, MR2021601}.  When $N>M^2$, the orthoplex bound provides an alternative means for characterizing known examples of Grassmannian frames ~\cite{MR1418961, MR0074013, bodmann2015achieving}, but this bound can only be achieved if $N \leq 2(M^2-1)$ ~\cite{MR1418961, MR0074013, bodmann2015achieving}.  

For the special case of unit-norm frames in $\mathbb C^2$, the isometric spherical embedding technique of Conway, Hardin, and Sloane~\cite{MR1418961} sends frame vectors to the unit sphere in $\mathbb R^3$.  As has been previously shown in \cite{MR2882247}, leveraging T\'{o}th's spherical cap packing bound for the unit sphere in $\mathbb R^3$ ~\cite{fejes1943covering} yields a lower bound for optimal coherence that is stronger than the orthoplex bound whenever $N>6$ (Theorem~\ref{th_toth}), and this saturates for the case $N=12$, as exemplified by a tight Grassmannian frame consisting of $12$ vectors in $\mathbb C^2$ (Example~\ref{ex_12_2}).  As with the cases of Grassmannian frames consisting of $N=4$ and $6$ vectors for $\mathbb C^2$, it is striking that this example also embeds perfectly into the vertices of a Platonic solid, an icosahedron in this case.  Furthermore, we observe that this example generates a \textem{complex projective 5-design} and is thus relevant to combinatoral and quantum information literature.

Because Grassmannian frames that achieve the Welch bound are tight and have angle sets of minimal cardinality, it is natural to ask the following questions:
\begin{question}\label{q_gr_tight}
Is every Grassmannian frame tight?
\end{question}
\begin{question}\label{q_gr_angles}
If $\Phi$ is a Grassmannian frame not characterized by the Welch bound, can we infer anything about the cardinality of its angle set or its spectral properties?
\end{question}

In ~\cite{BK06}, the authors answered Question~\ref{q_gr_tight}  in the negative for the real case by showing that Grassmannian frames consisting of $5$ vectors in $\mathbb R^3$ are  always equiangular  but never tight; furthermore, their result suggests that a plausible answer to Question~\ref{q_gr_angles} is that the cardinality of the angle set of a real Grassmannian frame should satisfy a minimality condition.

By considering two distinct examples of Grassmannian frames consisting of $5$ vectors in $\mathbb C^2$, we find the answers to these questions for the complex case to be more complicated than we anticipated. Strictly speaking, the answer to Question~\ref{q_gr_tight} for the complex case is also in the negative, because there exists a non-tight, $2$-angular Grassmannian frame with $5$ vectors for $\mathbb C^2$ (Example~\ref{ex_Gr_52_bi}) ; however, this question may be the wrong one to ask, because a tight, $3$-angular Grassmannian frame consisting of $5$ vectors over $\mathbb C^2$ also exists (Example~\ref{ex_Gr_52_tri}).  

Nevertheless, we find that one still encounters a certain amount of rigidity if one stipulates tightness in the frame's design; in particular, we prove that every tight Grassmannian frame consisting of $5$ vectors in $\mathbb C^2$ must have an angle set with cardinality greater than $2$ (Theorem~\ref{th_5_2}).
The proof of this depends on basic properties about $2$-angular, tight frames (Proposition~\ref{prop_equi} and Theorem~\ref{th_btfs_odd}) and the observation that whenever the spherical embedding technique of ~\cite{MR1418961}  is applied to a tight, unit-norm frame, then the embedded vectors in the higher-dimensional Euclidean sphere must sum to zero (Theorem~\ref{th_sph_emb_zero}).

In light of the coexistence of tight, $3$-angular and non-tight, $2$-angular Grassmannian frames in this scenario, we arrive at the following partial answer to the questions above.  There exist cases where, for a given fixed number of vectors and fixed dimension, a Grassmannian frame may be constructed in multiple ways, where there is some trade-off between its spectral properties (tightness) and geometric properties (angle set).

The remainder of this paper is outlined as follows. In Section~\ref{sec:prelim}, we establish notation and terminology and collect a few basics facts about frame theory.  In Section~\ref{sec:emb}, we recall the spherical embedding technique of ~\cite{MR1418961}  and show that it embeds tight frames into zero-summing vectors. In Section~\ref{sec:bounds}, we recall the Welch bound and orthoplex bound, and we use the spherical embedding technique along with a result of T\'{o}th  to improve upon these bounds for the case of $N>6$ vectors in $\mathbb C^2$.  In this section, we also discuss several examples of Grassmannian frames that achieve these bounds.  Finally, in Section~\ref{sec:gr_5_2}, we prove a few basic facts about $2$-angular tight frames and use these facts to prove that a tight Grassmannian frame consisting of $5$ vectors in $\mathbb C^2$ can never be $2$-angular.

\section{Preliminaries}\label{sec:prelim}
Let $\{e_j\}_{j=1}^M$ denote the canonical orthonormal basis for $\mathbb F^M$, where $\mathbb F = \mathbb R$ or $\mathbb C$ and let $I_M$ denote the $M \times M$ identity matrix. A set of vectors $\Phi = \{\phi_j\}_{j=1}^N \subset \mathbb F^M$ is a \textem{(finite) frame} if $\SPAN \{\phi_j\}_{j=1}^N = \mathbb F^M.$  
It is often convenient to identify a frame  $\Phi = \{\phi_j\}_{j=1}^N$ in terms of its \textem{synthesis matrix}
$$\Phi = \left[ \phi_1 \, \, \phi_2 \, \, ... \,\, \phi_N \right],$$
the $M \times N$ matrix with columns given by the \textem{frame vectors}.  Just as we have written
$\Phi = \{\phi_j\}_{j=1}^N $
and 
$\Phi = \left[ \phi_1 \, \, \phi_2 \, \, ... \,\, \phi_N \right]$ in the last sentence, we  proceed with the tacit understanding that $\Phi$ is both a matrix and a set of vectors.  Furthermore, we reserve the symbols $M$ and $N$ to refer to the dimension of the span of a frame and the cardinality of a frame, respectively.

A frame $\Phi = \{\phi_j\}_{j=1}^N$  is \textem{$A$-tight} if $\Phi \Phi^* = \sum_{j=1}^N \phi_j  \phi_j^*= A I_M$ for some $A>0$ and it is \textem{unit-norm} if each frame vector has norm $\|\phi_j\|=1$.  If $\Phi$ is unit-norm and $A$-tight, then $A=\frac{N}{M}$ because
$$N= \sum_{l=1}^M\sum_{j=1}^N |\langle e_l, \phi_j \rangle|^2 = \sum_{l=1}^M\sum_{j=1}^N \TR(\phi_j \phi_j^* e_l e_l^*) = A \sum_{l=1}^M \| e_l \|^2 = AM,$$
in which case also we have the identity
\begin{equation}\label{eq_tight_un}
\sum\limits_{l=1}^N | \langle \phi_j, \phi_l \rangle |^2  = \frac{N}{M},
\end{equation}
for every $j \in \{1,2,...,N \}$.

Given a unit-norm frame $\Phi = \{ \phi_j \}_{j=1}^N$, its \textem{frame angles} are the elements of the set 
$$
A_\Phi : =\{ |\langle \phi_j, \phi_l \rangle | : j \neq l \},
$$ i.e. \textem{the angle set of $\Phi$}, and we say that $\Phi$ is \textem{$K$-angular} if $|A_\Phi| =K$ for some $K \in \mathbb N$.   In the special cases that $K=1,2$ or $3$, we say that $\Phi$ is \textem{equiangular}, \textem{biangular} or \textem{triangular}, respectively.  

If $\Phi = \{ \phi_j \}_{j=1}^N$ is $K$-angular with frame angles $c_1, c_2,...,c_K$, then we say that $\Phi$ is \textem{equidistributed} if there exist $m_1, m_2, ..., m_K \in \mathbb N$ such that 
$$
\bigg| \Big\{ j' \in \{1,...,N\} : j' \neq j, |\langle \phi_j, \phi_{j'} \rangle | = c_k \Big\} \bigg| = m_k 
$$
for every $j \in \{1,2,...,N\}$ and every $k \in \{1,2,...,K\}$.  In this case, we call the positive integers $m_1,m_2,...,m_K$  the \textem{multiplicities} of $\Phi$ and remark that $\sum_{j=1}^K m_j= N-1$.

We let $\UNFs{F}$ denote the space of unit norm frames consisting of $N$ vectors in $\mathbb F^M$.
Given $\Phi = \{ \phi_j \}_{j=1}^N \in \UNFs{F}$,  its \textem{coherence} is defined by
$$\mu(\Phi) = \max\limits_{j \neq l} |\langle \phi_j, \phi_l \rangle |$$
and we define the \textem{Grassmanian constant for the pair $(N,M)$} by
$$\Gnk{F} = \inf\limits_{{\Phi} \in \UNFs{F}} \max\limits_{j \neq l} |\langle \phi_j, \phi_l \rangle | .$$ 
We say that $\Phi \in \UNFs{F}$ is a \textem{Grassmannian frame} if
$$
\mu(\Phi) = \Gnk{F}. 
$$

\section{Spherical Embedding}\label{sec:emb}

In ~\cite{MR1418961}, the authors observed that a unit-norm frame can be isometrically embedded into a sphere in some high dimensional real Hilbert space.  

\begin{thm}\label{th_sph_emb}[Conway et al.,~\cite{MR1418961}]
Let $D=M^2-1$ if $\mathbb F = \mathbb C$ or let $D=\frac{(M+2)(M-1)}{2}$ if $\mathbb F = \mathbb R$. If $\Phi = \{\phi_j\}_{j=1}^N \in \UNFs{F}$, then the frame vectors can be isometrically embedded into the unit sphere in $\mathbb R^D$ via
$$\phi_j \mapsto y_j \in \mathbb R^{D}$$ such that, for all $j, l \in \{1,...,N\}$, we have
\begin{equation*}
| \langle \phi_j, \phi_l \rangle |^2 = \frac{1}{M} + \frac{M-1}{M} \langle y_j, y_l \rangle.
\end{equation*}
\end{thm}

\begin{proof}
Let $\Phi = \{\phi_j\}_{j=1}^N \in \UNFs{F}$.  The frame vectors of $\Phi$ can be  embedded into the "traceless" subspace of the $M \times M$ self-adjoint (or symmetric) matrices via the mapping
$$ \phi_j \mapsto \phi_j   \phi_j^* - \frac{1}{M }I_M,$$
which is isomorphic to $\mathbb R^{D}$ by a dimension counting argument.  In particular, these embedded vectors all lie on a sphere of radius $\sqrt{\frac{M-1}{M}}$, because the Hilbert Schmidt norm gives $\| \phi_j   \phi_j^* - \frac{1}{M }I_M\|_{H.S.}^2 = \frac{M-1}{M}$ for every $j \in \{1,...,N\}$, and this embdedding is distance preserving because for $j \neq l$
\begin{align*}
\| \phi_j   \phi_j^* - \phi_l   \phi_l^*\|_{H.S.}^2 
&=  2\left(1- tr( \phi_j   \phi_j^* \phi_l   \phi_l^* )\right) \\
&= 2\left(1- |\langle \phi_j, \phi_l \rangle|^2 )\right). \\
\end{align*}
By identifying $\phi_j   \phi_j^* - \frac{1}{M }I_M$ and $\phi_j   \phi_j^* - \frac{1}{M }I_M$ with vectors $\tilde y_j, \tilde y_l \in \mathbb R^{D}$ on a sphere of radius $\sqrt{\frac{M-1}{M}}$
and using that
$$
\|y_j - y_l\|^2 = 2\frac{M-1}{M} (1 - \langle \tilde y_j , \tilde y_l \rangle)
,
$$
we can rewrite this equation as
$$
 |\langle \phi_ j, \phi_l \rangle|^2 = \frac{1}{M} + \frac{M-1}{M}  \langle y_j , y_l \rangle,
$$
where $y_j$ and $y_l$ are the unit vectors in the direction of $\tilde y_j$ and $\tilde y_l$, respectively.
\end{proof}

We observe that whenever a unit-norm, tight frame is embedded into a higher dimensional sphere as above, then the embedded vectors are also zero-summing.

\begin{cor}\label{th_sph_emb_zero}
Let $D=M^2-1$ if $\mathbb F = \mathbb C$ or let $D=\frac{(M+2)(M-1)}{2}$ if $\mathbb F = \mathbb R$. If $\Phi = \{\phi_j\}_{j=1}^N \in \UNFs{F}$ is a tight frame, then the frame vectors can be isometrically embedded into the unit sphere in $\mathbb R^D$ via
$$\phi_j \mapsto y_j \in \mathbb R^{D}$$ such that, for all $j, l \in \{1,...,N\}$, we have
\begin{equation}\label{eq_sph_emb_angles}
| \langle \phi_j, \phi_l \rangle |^2 = \frac{1}{M} + \frac{M-1}{M} \langle y_j y_l \rangle
\end{equation}
and
\begin{equation}\label{eq_sph_emb_zero}
\sum\limits_{j=1}^N y_j = 0.
\end{equation}
\end{cor}

\begin{proof}
Let $\Phi = \{\phi_j\}_{j=1}^N \in \UNFs{F}$.  As in the proof of Theorem~\ref{th_sph_emb}, we  embed the frame vectors via
$$ \phi_j \mapsto \phi_j   \phi_j^* - \frac{1}{M }I_M.$$
Summing over these matrices and using the tightness property, we have
$$
\sum\limits_{j=1}^N \left( \phi_j   \phi_j^* - \frac{1}{M }I_M  \right) = \frac{N}{M} I_M - \frac{N}{M} I_M= 0_{M},
$$
where $0_{M}$ denotes the $M \times M$ zero matrix.  The claim then follows by mimicing the proof of  Theorem~\ref{th_sph_emb}.
\end{proof}

\section{Lower Bounds for the Grassmannian constant}\label{sec:bounds}

The optimal lower bound for the Grassmannian constant is the Welch bound ~\cite{welch:bound},
$$
  \Gnk{F} \ge \sqrt{\frac{N-M}{M(N-1)}}.
$$
A Grassmannian frame  achieves this lower bound if and only if it is an equiangular, tight frame ~\cite{MR1984549, MR2711357}, but it is well-known that this cannot occur when $N>M^2$ if $\mathbb F =  \mathbb C$ or $N>\frac{M(M+1)}{2}$ if $\mathbb F =  \mathbb R$ ~\cite{MR1418961, MR2021601}.

 By applying Theorem~\ref{th_sph_emb} to a sphere-packing result of Rankin ~\cite{MR0074013}, the authors of ~\cite{MR1984549} (see also ~\cite{henkel05, bodmann2015achieving})  extrapolated a lower bound for $\Gnk{F}$ that is stronger than the Welch bound whenever $N>M^2$ or $N> \frac{M(M+1)}{2}$ for $\mathbb F = \mathbb C$ or $\mathbb R$, respectively.

\begin{thm}\label{th_ortho} [Orthoplex bound, ~\cite{MR0074013,MR1984549,henkel05, bodmann2015achieving}]
Let $D=M^2-1$ if $\mathbb F = \mathbb C$ or let $D=\frac{(M+2)(M-1)}{2}$ if $\mathbb F = \mathbb R$. 
If $N>D+1$, then
$$
\Gnk{F} \geq \frac{1}{\sqrt{M}} \, .
$$
\end{thm}

The following example comes from ~\cite{bodmann2015achieving}, where the authors constructed infinite families of tight, complex Grassmannian frames with coherence equal to the orthoplex bound, which they termed \textem{orthoplectic Grassmannian frames}.  Of particular interest to us, it is a tight, triangular Grassmannian frame in $\Omega_{5,2} ( \mathbb C)$.

\begin{ex}\label{ex_Gr_52_tri}
Let $\omega = e^{2 \pi i /3}$ and let 
$$
\Phi = \frac{1}{\sqrt{2}} 
\left[   
\begin{array}{ccccc}
\sqrt{2} & 0 & 1 & 1 & 1\\
0 & \sqrt{2} & 1 & \omega & \omega^2
\end{array}
\right].
$$
It is straightforward to check that $\Phi$ is tight and has frame angles $c_1 = \frac{1}{\sqrt{2}}, c_2 = \frac{1}{2}$ and $c_3 =0$, so it is triangular and it is a Grassmannian frame by Theorem~\ref{th_ortho}.
\end{ex}

The next example also achieves the orthoplex bound, so it is also a Grassmannian frame in 
$\Omega_{5,2} ( \mathbb C)$; however, unlike the preceding example, this one is biangular but not tight.
\begin{ex}\label{ex_Gr_52_bi}
Let 
$$
\Phi = \frac{1}{\sqrt{2}} 
\left[   
\begin{array}{ccccc}
\sqrt{2} & 1 & 1 & 1 & 1\\
0 & -1 & 1 & i & -i
\end{array}
\right].
$$
It is straightforward to check that $\Phi$ has frame angles $c_1 = \frac{1}{\sqrt{2}}$ and $c_2 = 0$, so it is biangular and it is a Grassmannian frame by Theorem~\ref{th_ortho}.  However, $\Phi \Phi^* \neq \frac{5}{2} I_5$, so it is not tight.
\end{ex}

Just as the Welch bound cannot saturate when a frame's cardinality is too large, it also known that a Grassmannian frame's coherence must be greater than
 the orthoplex bound when $N>2(M^2-1)$ in the complex case or $N>(M+2)(M-1)$ in the real case~\cite{MR0074013,MR1984549,henkel05, bodmann2015achieving}.  Thus, the orthoplex bound can also be improved when there are too many vectors. 
 In the special case that $\mathbb F = \mathbb C$ and $M=2$, the embedding from Theorem~\ref{th_sph_emb} sends points from the sphere in $\mathbb C^2$ to points on the sphere in $\mathbb R^3$, so the classical spherical cap packing result of T\'{o}th ~\cite{fejes1943covering} leads to an improved lower bound for $\mu_{N,2} (\mathbb{C})$ when $N\geq 7$, as has been previously noted in \cite{MR2882247}.

\begin{thm}\label{th_toth}[T\'{o}th's Bound, ~\cite{fejes1943covering}]
Let $N\geq 3$.  If $\{x_j\}_{j=1}^N$ is a set of unit vectors in $\mathbb R^3$, then
$$
\max\limits_{j \neq l} \langle x_j, x_l \rangle \geq \frac{1}{2} \csc^2\left(\frac{N \pi}{6(N-2)} \right) - 1
.
$$
\end{thm}

\begin{thm}\label{cor_toth}[\cite{MR2882247}]
If $N \geq 3$, then 
$$
\mu_{N,2} (\mathbb C) \geq  \frac{1}{2} \csc \left(\frac{N \pi}{6(N-2)} \right)  .
$$
\end{thm}

\begin{proof}
Let $\Phi = \{ \phi_j \}_{j=1}^N \in \Omega_{N,2} (\mathbb C).$  By Theorem~\ref{th_sph_emb}, there exist points $\{y_j\}_{j=1}^N$ on the unit sphere in $\mathbb R^3$ such that 
$$
| \langle \phi_j, \phi_l \rangle |^2 = \frac{1}{2} + \frac{1}{2} \langle y_j, y_l\rangle.
$$
The claim then follows from Theorem~\ref{th_toth} because
\begin{align*}
\max\limits_{j \neq l} | \langle \phi_j, \phi_l \rangle |^2 
&= \frac{1}{2} + \frac{1}{2} \max\limits_{j \neq l} \langle y_j, y_l \rangle \\
&\geq \frac{1}{2} + \frac{1}{2} \left[  \frac{1}{2} \csc^2\left(\frac{N \pi}{6(N-2)} \right) - 1  \right]. \\
\end{align*}
\end{proof}

T\'{o}th's bound from Theorem~\ref{th_toth} saturates when $N=3,4,6,$ and $12$ ~\cite{fejes1943covering}.  When $N=3$, the bound is obtained by the three vertices of an equilateral triangle centered at the origin in $\mathbb R^3$.  For $N=4$, the bound is obtained by the vertices of a regular $3$-simplex (i.e. a tetrahedron) centered at the origin.  For $N=6$, the bound is obtained by the vertices of an orthoplex (i.e. an $\ell^1$-ball or octahedron) centered at the origin and the the case $N=12$ corresponds to the twelve vertices of an icosahedron centered at the origin.  Furthermore, for the cases $N=3,4,$ and $6$, it is known that there exist tight Grassmannian frames in $\Omega_{N,2} (\mathbb C)$ that not only achieve the lower bound in Theorem~\ref{cor_toth} but  embed perfectly into the vertices of an equilateral triangle~\cite{MR1418961}, regular tetrahedron~\cite{MR2059685}, and regular octahedron~\cite{MR2337670}, respectively.  Next, we exhibit an example of a tight Grassmannian frame in $\Omega_{12,2} (\mathbb C)$ that embeds perfectly into the vertices of a regular icosahedron.

\begin{ex}\label{ex_12_2} [Icosaplectic Grassmannian Frame]
Let $a=\sqrt{\frac{5+\sqrt{5} }{10}}$ and $b=\sqrt{1 -a^2}$ and let $\eta = e^{2 \pi i /5}$ and  $\omega = e^{2 \pi i/10}$ be the primitive $5$th and $10$th roots of unity, respectively.
A straightforward computation shows that
$$
\Phi = 
\left[
\begin{array}{cccccccccccc}
\scriptscriptstyle 1 & \scriptscriptstyle 0 & \scriptscriptstyle -b & \scriptscriptstyle -b & \scriptscriptstyle -b & \scriptscriptstyle -b  & \scriptscriptstyle -b& \scriptscriptstyle a \scriptscriptstyle \omega  & \scriptscriptstyle a \scriptscriptstyle \omega \scriptscriptstyle \eta & \scriptscriptstyle a \scriptscriptstyle \omega \scriptscriptstyle \eta^2 & \scriptscriptstyle a \scriptscriptstyle \omega \scriptscriptstyle \eta^3  &\scriptscriptstyle a \scriptscriptstyle \omega \scriptscriptstyle \eta^4  \\
\scriptscriptstyle 0 &\scriptscriptstyle 1 & 
 \scriptscriptstyle -a \scriptscriptstyle \omega  & \scriptscriptstyle -a \scriptscriptstyle \omega \scriptscriptstyle \eta & \scriptscriptstyle -a \scriptscriptstyle \omega \scriptscriptstyle \eta^2 & \scriptscriptstyle -a \scriptscriptstyle\omega \scriptscriptstyle\eta^3  &\scriptscriptstyle -a \scriptscriptstyle\omega \scriptscriptstyle\eta^4 &
\scriptscriptstyle b &\scriptscriptstyle  b &\scriptscriptstyle b &\scriptscriptstyle b &\scriptscriptstyle b
\end{array}
\right]
$$
is an equidistributed, triangular, tight frame in  $\Omega_{12,2} (\mathbb C)$ with frame angles 
$
c_1 = a > c_2 =b > c_3 =0 $
and corresponding multiplicities $m_1 =5 , m_2 =5 $ and $m_3 =1$.  It follows from elementary trigonometry that the lower bound in Theorem~\ref{cor_toth} equals $c_1$, showing that $\Phi$ is a Grassmannian frame in $\Omega_{12,2} ( \mathbb C)$. If $Y = \{y_j\}_{j=1}^{12}$ denotes the unit vectors in $\mathbb R^3$ obtained via the embedding from Theorem~\ref{th_sph_emb}, then another  computation using the identity
$$
| \langle \phi_j, \phi_l \rangle |^2 = \frac{1}{2} + \frac{1}{2} \langle y_j, y_l\rangle
$$
shows that the vectors of $Y$ must correspond to the vertices of a regular icosahedron.
\end{ex}

Finally, we remark that the frame from Example~\ref{ex_12_2} has some relevance for the combinatorial and quantum information literature.  It is simple to check that
$$
\frac{1}{(12)^2}\sum\limits_{j,l=1}^N |\langle \phi_j, \phi_l \rangle |^{10} = \frac{1}{6},
$$
so $\Phi$ generates an (equally) \textem{weighted complex projective $5$-design} by Theorem~{2.3} of ~\cite{MR2337670}.  Such objects are known to be optimal for linear quantum state determination with respect to a fixed number of measurements ~\cite{MR2269701, MR2337670}.  

\section{Grassmannian Frames for $\mathbb C^2$ Consisting of $5$ vectors }\label{sec:gr_5_2}

In this section, we show that although $\mu_{5,2}(\mathbb C)$ can be achieved by both biangular and triangular frames, as in Examples~\ref{ex_Gr_52_bi} and \ref{ex_Gr_52_tri}, there is a necessaray trade-off between the cardinality of the angle set and tightness.  In order to do this, we collect a few basic facts about \textem{biangular, tight frames (BTFs)}.

First, we show that every BTF is equidistributed.  A specialized version of this result was shown for the case of {\it $2$-distant tight frames} in \cite{MR3325226}.

\begin{prop}\label{prop_equi}
If $\Phi=\{\phi_j \}_{j =1}^N$ is a biangular, tight frame for $\mathbb F^M$, then $\Phi$ is equidistributed.
\end{prop}
\begin{proof}
If $c_1, c_2$ are the frame angles of $\Phi$, then Equation~\ref{eq_tight_un} implies that for each $j \in \{1,2,...,N\}$, there exists a pair of positive integers $m_{1, j}$ and $m_{2,j}$ such that $m_{1, j} + m_{2,j} =N-1$ and
$$
m_{1,j} c_1^2 + m_{2,j} c_2^2
=
\sum\limits_{l =1, l \neq j}^N | \langle \phi_j, \phi_l \rangle |^2 
=
\| \phi_j \|^2 \left( \frac{N}{M}- 1 \right)
=\frac{N-M}{M},
$$
where the last equality follows from the unit-norm property.  Next, for $j,l \in \{1,2,...,N\}$, we compute:
\begin{align*}
(N-1)c_1^2+m_{2,j}(c_2^2-c_1^1)&=(m_{1,j}+m_{2,j})c_1^2
+ m_{2,j}(c_2^2-c_1^2)\\
&= m_{1,j}c_1^2+m_{2,j}c_2^2\\
&= m_{1,l}c_1^2+m_{2,l}c_2^2\\
&=(m_{1,l}+m_{2,l})c_1^2+m_{2,l}(c_2^2-c_2^1)\\
&= (N-1)c_1^2+m_{2,l}(c_2^2-c_1^2).
\end{align*}
Since $c_1\not= c_2$, it follows that $m_{2,j}=m_{2,l}$, which implies that
$m_{1,j}=m_{1,l}$.
\end{proof}

It is worth noting that, in general, K-angular tight frames are not equidistributed, so BTFs and ETFs are quite special in this regard.  For instance, the frame from Example~\ref{ex_Gr_52_tri} is tight and $3$-angular, but it is not equidistributed.    The second observation we need about BTFs concerns the multiplicities of their frame angles when $N$ is odd.

\begin{lemma}\label{lem_even}
Let $N \in \mathbb N$ be odd.  If $\Phi=\{\phi_j \}_{j =1}^N$ is a biangular, equidistributed frame for $\mathbb F^M$ with frame angles $c_1, c_2$ and corresponding multiplicities $m_1, m_2$, then $m_1$ and $m_2$ are both even.
\end{lemma}

\begin{proof}

 Let $M = \left( |\langle \phi_l, \phi_j \rangle | \right)_{j,l=1}^N$, the matrix obtained by taking the absolute values of the entries of the Gram matrix of $\Phi$.
If $m_1$ is odd, then $N m_1$ is odd, so $c_1$ occurs an odd number of times among the off-diagonal entries of $M$.   However, $M$ is symmetric,  so the number of occurrences of $c_1$  above the diagonal of $M$ equals the number of occurrences of $c_1$ below the diagonal,which implies that $N m_1$ is even, a contradiction.  Therefore, $m_1$ is even, so the fact that $N-1 = m_1 + m_2$ is even implies that $m_2$ is also even.
\end{proof}

\begin{thm}\label{th_btfs_odd}
If $N$ is odd and $\Phi=\{\phi_j \}_{j =1}^N$ is a biangular, tight frame for $\mathbb F^M$, then $\Phi$ is equidistributed with even multiplicities.
\end{thm}

\begin{proof}
This follows directly from Proposition~\ref{prop_equi} and Lemma~\ref{lem_even}.
\end{proof}

Finally, we show that a tight, biangular Grassmannian frame can never exist in $\Omega_{5,2} ( \mathbb C)$.

\begin{thm}\label{th_5_2}
If  $\Phi$ is a tight Grassmannian frame in $\Omega_{5,2} (\mathbb C)$, then $|A_\Phi| \geq 3$.
\end{thm}

\begin{proof}
First, note that Examples~\ref{ex_Gr_52_tri} and \ref{ex_Gr_52_bi} show that the lower bound in Theorem~\ref{th_ortho} is saturated in this setting.  In particular, we know that $\mu_{5,2} (\mathbb C) = \frac{1}{\sqrt{2}}$.

By way of contradiction, suppose that $|A_\Phi|<3$.  If $|A_\Phi|=1$, then $\Phi$ is an ETF, so it achieves the Welch bound, which means
$
\mu(\Phi) = \sqrt{\frac{3}{8}} < \frac{1}{\sqrt{2}},
$ 
a contradiction.  
Therefore, it must be that $|A_\Phi| =2$, meaning $\Phi$ is a BTF.

 Let $c_1,c_2$ denote the frame angles of $\Phi$.  Because $\Phi$ is a Grassmannian frame, we may assume with no loss of generality that $c_1=\frac{1}{\sqrt{2}}.$  By Corollary~\ref{th_btfs_odd}, $\Phi$ is equidistributed with multiplicities $m_1=m_2=2$.  Because $\Phi$ is a $\frac{5}{2}$-tight, unit-norm frame, Equation~\ref{eq_tight_un} becomes
$$
\frac{5}{2} = 1 + 2 c_1^2 + 2 c_2^2 = 1 + 2 \left(\frac{1}{\sqrt{2}} \right)^2 + 2 c_2^2,
$$
which implies that $c_2 = \frac{1}{2}$.

Next, let $Y = \{y_j\}_{j=1}^5$ be the zero-summing unit vectors obtained by embedding the vectors of $\Phi$ into $\mathbb R^3$, as in Corollary~\ref{th_sph_emb_zero}.  Because $\Phi$ is equidistributed with multiplicities $m_1=m_2=2$, it follows from Equation~\ref{eq_sph_emb_angles} that the statement
\begin{equation*}
\mathcal P_j : 
\left| \{ \langle y_j, y_l \rangle = 0 : l \neq j   \} \right| = \left| \{ \langle y_j, y_l \rangle= -1/2 = 0 : l \neq j   \} \right| = 2
\end{equation*}
must be true for all $j \in \{1,2,3,4,5\}$.  We will show that this  contradicts the zero-summing property of $Y$.

After an appropriate choice of unitary rotation, there is no loss of generality in assuming that $y_1 = e_1$.  The statement $\mathcal P_1$ implies that $y_1$ is orthogonal to two of the vectors of $Y$ and it has inner product $-\frac{1}{2}$ with the other two; therefore, we may further assume without loss of generality that, after an apropriate rotation, $y_2 = e_2$ and  that $\langle y_1, y_3 \rangle =0$ and $\langle y_1, y_4 \rangle = \langle y_1, y_5 \rangle = -\frac{1}{2}$.  Viewing $Y$ as a $3 \times 5$ matrix, these assumptions mean that its first row and it first two columns are completely determined.
$$
Y = 
\left[
\begin{array}{ccccc}
1 & 0 & 0 & -\frac{1}{2} & -\frac{1}{2} \\
0 & 1 & * & * & * \\
0 & 0 & * & * & * \\
\end{array}
\right].  
$$
We cannot have $\langle y_2, y_3 \rangle =0$, because then the statements $\mathcal P_1, \mathcal  P_2$ and $\mathcal P_3$ force
$$
\langle y_1, y_4 \rangle = \langle y_2, y_4 \rangle = \langle y_3, y_4 \rangle,
$$
which in turn contradicts $\mathcal P_4$.  Therefore,  $\langle y_2, y_3 \rangle = -1/2$.
$$
Y = 
\left[
\begin{array}{ccccc}
1 & 0 & 0 & -\frac{1}{2} & -\frac{1}{2} \\
0 & 1 & -\frac{1}{2} & * & * \\
0 & 0 & * & * & * \\
\end{array}
\right].  
$$
The statement $\mathcal P_2$ then implies that either (i) $\langle y_2, y_4 \rangle =0$ and  $\langle y_2, y_5 \rangle =-1/2$ or  (ii) $\langle y_2, y_4 \rangle =-1/2$ and  $\langle y_2, y_5 \rangle =0$.  Since it is clear that these two cases are symmetric, we assume with no loss in generality that  $\langle y_2, y_4 \rangle =0$ and  $\langle y_2, y_5 \rangle =-1/2$.
$$
Y = 
\left[
\begin{array}{ccccc}
1 & 0 & 0 & -\frac{1}{2} & -\frac{1}{2} \\
0 & 1 & -\frac{1}{2} & 0 & -\frac{1}{2} \\
0 & 0 & * & * & * \\
\end{array}
\right].  
$$
Finally, the unit-norm condition means that the remaining entries of $Y$ must satisfy
$$
Y = 
\left[
\begin{array}{ccccc}
1 & 0 & 0 & -\frac{1}{2} & -\frac{1}{2} \\
0 & 1 & -\frac{1}{2} & 0 & -\frac{1}{2} \\
0 & 0 & \pm \frac{\sqrt{3}}{2}  & \pm \frac{\sqrt{3}}{2} & \pm \frac{\sqrt{2}}{2} \\
\end{array}
\right], 
$$
but this contradicts the zero-summing condition, because there is no choice of signs for which
$$
\pm \frac{\sqrt{3}}{2}  \pm \frac{\sqrt{3}}{2}   \pm \frac{\sqrt{2}}{2} = 0.
$$
Therefore, $\Phi$ cannot exist, which completes the proof.
\end{proof}

\section*{Acknowledgments}
The authors were supported by
 NSF DMS 1307685, NSF ATD 1321779, and ARO W911NF-16-1-0008.

\section*{References}

\bibliography{C2_grass_bib}

\begin{thebibliography}{10}
\expandafter\ifx\csname url\endcsname\relax
  \def\url#1{\texttt{#1}}\fi
\expandafter\ifx\csname urlprefix\endcsname\relax\def\urlprefix{URL }\fi
\expandafter\ifx\csname href\endcsname\relax
  \def\href#1#2{#2} \def\path#1{#1}\fi

\bibitem{MR2662471}
A.~J. Scott, M.~Grassl,
  \href{http://dx.doi.org.ezproxy.lib.uh.edu/10.1063/1.3374022}{Symmetric
  informationally complete positive-operator-valued measures: a new computer
  study}, J. Math. Phys. 51~(4) (2010) 042203, 16.
\newblock \href {http://dx.doi.org/10.1063/1.3374022}
  {\path{doi:10.1063/1.3374022}}.
\newline\urlprefix\url{http://dx.doi.org.ezproxy.lib.uh.edu/10.1063/1.3374022}

\bibitem{MR2059685}
J.~M. Renes, R.~Blume-Kohout, A.~J. Scott, C.~M. Caves,
  \href{http://dx.doi.org.ezproxy.lib.uh.edu/10.1063/1.1737053}{Symmetric
  informationally complete quantum measurements}, J. Math. Phys. 45~(6) (2004)
  2171--2180.
\newblock \href {http://dx.doi.org/10.1063/1.1737053}
  {\path{doi:10.1063/1.1737053}}.
\newline\urlprefix\url{http://dx.doi.org.ezproxy.lib.uh.edu/10.1063/1.1737053}

\bibitem{MR2778089}
D.~M. Appleby, S{IC}-{POVMS} and {MUBS}: geometrical relationships in prime
  dimension, in: Foundations of probability and physics---5, Vol. 1101 of AIP
  Conf. Proc., Amer. Inst. Phys., New York, 2009, pp. 223--232.

\bibitem{5946867}
D.~Mixon, C.~Quinn, N.~Kiyavash, M.~Fickus, Equiangular tight frame
  fingerprinting codes, in: Acoustics, Speech and Signal Processing (ICASSP),
  2011 IEEE International Conference on, 2011, pp. 1856--1859.
\newblock \href {http://dx.doi.org/10.1109/ICASSP.2011.5946867}
  {\path{doi:10.1109/ICASSP.2011.5946867}}.

\bibitem{MR2028016}
D.~J. Love, R.~W. Heath, Jr.,
  \href{http://dx.doi.org.ezproxy.lib.uh.edu/10.1109/TIT.2003.817466}{Grassmannian
  beamforming for multiple-input multiple-output wireless systems}, IEEE Trans.
  Inform. Theory 49~(10) (2003) 2735--2747.
\newblock \href {http://dx.doi.org/10.1109/TIT.2003.817466}
  {\path{doi:10.1109/TIT.2003.817466}}.
\newline\urlprefix\url{http://dx.doi.org.ezproxy.lib.uh.edu/10.1109/TIT.2003.817466}

\bibitem{MR2921716}
T.~R. Hoffman, J.~P. Solazzo,
  \href{http://dx.doi.org.ezproxy.lib.uh.edu/10.1016/j.laa.2012.01.024}{Complex
  equiangular tight frames and erasures}, Linear Algebra Appl. 437~(2) (2012)
  549--558.
\newblock \href {http://dx.doi.org/10.1016/j.laa.2012.01.024}
  {\path{doi:10.1016/j.laa.2012.01.024}}.
\newline\urlprefix\url{http://dx.doi.org.ezproxy.lib.uh.edu/10.1016/j.laa.2012.01.024}

\bibitem{MR2021601}
R.~B. Holmes, V.~I. Paulsen,
  \href{http://dx.doi.org.ezproxy.lib.uh.edu/10.1016/j.laa.2003.07.012}{Optimal
  frames for erasures}, Linear Algebra Appl. 377 (2004) 31--51.
\newblock \href {http://dx.doi.org/10.1016/j.laa.2003.07.012}
  {\path{doi:10.1016/j.laa.2003.07.012}}.
\newline\urlprefix\url{http://dx.doi.org.ezproxy.lib.uh.edu/10.1016/j.laa.2003.07.012}

\bibitem{MR2149656}
B.~G. Bodmann, V.~I. Paulsen,
  \href{http://dx.doi.org.ezproxy.lib.uh.edu/10.1016/j.laa.2005.02.016}{Frames,
  graphs and erasures}, Linear Algebra Appl. 404 (2005) 118--146.
\newblock \href {http://dx.doi.org/10.1016/j.laa.2005.02.016}
  {\path{doi:10.1016/j.laa.2005.02.016}}.
\newline\urlprefix\url{http://dx.doi.org.ezproxy.lib.uh.edu/10.1016/j.laa.2005.02.016}

\bibitem{fickus2016tremain}
M.~Fickus, J.~Jasper, D.~G. Mixon, J.~Peterson, Tremain equiangular tight
  frames, arXiv preprint arXiv:1602.03490.

\bibitem{MR1984549}
T.~Strohmer, R.~W. Heath, Jr.,
  \href{http://dx.doi.org.ezproxy.lib.uh.edu/10.1016/S1063-5203(03)00023-X}{Grassmannian
  frames with applications to coding and communication}, Appl. Comput. Harmon.
  Anal. 14~(3) (2003) 257--275.
\newblock \href {http://dx.doi.org/10.1016/S1063-5203(03)00023-X}
  {\path{doi:10.1016/S1063-5203(03)00023-X}}.
\newline\urlprefix\url{http://dx.doi.org.ezproxy.lib.uh.edu/10.1016/S1063-5203(03)00023-X}

\bibitem{MR3150919}
J.~Jasper, D.~G. Mixon, M.~Fickus,
  \href{http://dx.doi.org.ezproxy.lib.uh.edu/10.1109/TIT.2013.2285565}{Kirkman
  equiangular tight frames and codes}, IEEE Trans. Inform. Theory 60~(1) (2014)
  170--181.
\newblock \href {http://dx.doi.org/10.1109/TIT.2013.2285565}
  {\path{doi:10.1109/TIT.2013.2285565}}.
\newline\urlprefix\url{http://dx.doi.org.ezproxy.lib.uh.edu/10.1109/TIT.2013.2285565}

\bibitem{MR2277977}
D.~Kalra,
  \href{http://dx.doi.org.ezproxy.lib.uh.edu/10.1016/j.laa.2006.05.008}{Complex
  equiangular cyclic frames and erasures}, Linear Algebra Appl. 419~(2-3)
  (2006) 373--399.
\newblock \href {http://dx.doi.org/10.1016/j.laa.2006.05.008}
  {\path{doi:10.1016/j.laa.2006.05.008}}.
\newline\urlprefix\url{http://dx.doi.org.ezproxy.lib.uh.edu/10.1016/j.laa.2006.05.008}

\bibitem{MR2015832}
B.~Et-Taoui,
  \href{http://dx.doi.org.ezproxy.lib.uh.edu/10.1016/S0019-3577(02)80027-1}{Equiangular
  lines in {${\bf C}^r$}. {II}}, Indag. Math. (N.S.) 13~(4) (2002) 483--486.
\newblock \href {http://dx.doi.org/10.1016/S0019-3577(02)80027-1}
  {\path{doi:10.1016/S0019-3577(02)80027-1}}.
\newline\urlprefix\url{http://dx.doi.org.ezproxy.lib.uh.edu/10.1016/S0019-3577(02)80027-1}

\bibitem{MR2736147}
P.~Singh,
  \href{http://dx.doi.org.ezproxy.lib.uh.edu/10.1016/j.laa.2010.07.027}{Equiangular
  tight frames and signature sets in groups}, Linear Algebra Appl. 433~(11-12)
  (2010) 2208--2242.
\newblock \href {http://dx.doi.org/10.1016/j.laa.2010.07.027}
  {\path{doi:10.1016/j.laa.2010.07.027}}.
\newline\urlprefix\url{http://dx.doi.org.ezproxy.lib.uh.edu/10.1016/j.laa.2010.07.027}

\bibitem{MR2350682}
M.~A. Sustik, J.~A. Tropp, I.~S. Dhillon, R.~W. Heath, Jr.,
  \href{http://dx.doi.org.ezproxy.lib.uh.edu/10.1016/j.laa.2007.05.043}{On the
  existence of equiangular tight frames}, Linear Algebra Appl. 426~(2-3) (2007)
  619--635.
\newblock \href {http://dx.doi.org/10.1016/j.laa.2007.05.043}
  {\path{doi:10.1016/j.laa.2007.05.043}}.
\newline\urlprefix\url{http://dx.doi.org.ezproxy.lib.uh.edu/10.1016/j.laa.2007.05.043}

\bibitem{MR3005267}
F.~Sz{{\"o}}ll{\H{o}}si,
  \href{http://dx.doi.org.ezproxy.lib.uh.edu/10.1016/j.laa.2011.05.034}{Complex
  {H}adamard matrices and equiangular tight frames}, Linear Algebra Appl.
  438~(4) (2013) 1962--1967.
\newblock \href {http://dx.doi.org/10.1016/j.laa.2011.05.034}
  {\path{doi:10.1016/j.laa.2011.05.034}}.
\newline\urlprefix\url{http://dx.doi.org.ezproxy.lib.uh.edu/10.1016/j.laa.2011.05.034}

\bibitem{Tremain_08}
J.~C. Tremain, \href{http://arxiv.org/abs/0811.2779}{Concrete constructions of
  real equiangular line sets}, arXiv e-print, arXiv:0811.2779\href
  {http://arxiv.org/abs/0811.2779} {\path{arXiv:0811.2779}}.
\newline\urlprefix\url{http://arxiv.org/abs/0811.2779}

\bibitem{MR2235693}
P.~Xia, S.~Zhou, G.~B. Giannakis,
  \href{http://dx.doi.org.ezproxy.lib.uh.edu/10.1109/TIT.2005.846411}{Achieving
  the {W}elch bound with difference sets}, IEEE Trans. Inform. Theory 51~(5)
  (2005) 1900--1907.
\newblock \href {http://dx.doi.org/10.1109/TIT.2005.846411}
  {\path{doi:10.1109/TIT.2005.846411}}.
\newline\urlprefix\url{http://dx.doi.org.ezproxy.lib.uh.edu/10.1109/TIT.2005.846411}

\bibitem{Zau}
G.~Zauner, Quantendesigns - {G}rundz{\"u}ge einer nichtkommutativen
  {D}esigntheorie, 1999, dissertation (Ph.D.)--University Wien (Austria).

\bibitem{Fickus:2015aa}
M.~Fickus, D.~G. Mixon, \href{http://arxiv.org/abs/1504.00253}{Tables of the
  existence of equiangular tight frames}, arXiv e-print, arXiv:1504.00253\href
  {http://arxiv.org/abs/1504.00253} {\path{arXiv:1504.00253}}.
\newline\urlprefix\url{http://arxiv.org/abs/1504.00253}

\bibitem{szollHosi2014all}
F.~Sz{\"o}ll{\H o}si, \href{http://arxiv.org/abs/1402.6429}{All complex
  equiangular tight frames in dimension 3}, arXiv preprint, arXiv
  1402.6429\href {http://arxiv.org/abs/1402.6429} {\path{arXiv:1402.6429}}.
\newline\urlprefix\url{http://arxiv.org/abs/1402.6429}

\bibitem{bgb15}
B.~Bodmann, J.~Haas, \href{http://dx.doi.org/10.1007/s00041-015-9408-z}{Frame
  potentials and the geometry of frames}, Journal of Fourier Analysis and
  Applications (2015) 1--40\href {http://dx.doi.org/10.1007/s00041-015-9408-z}
  {\path{doi:10.1007/s00041-015-9408-z}}.
\newline\urlprefix\url{http://dx.doi.org/10.1007/s00041-015-9408-z}

\bibitem{bodmann2015achieving}
B.~G. Bodmann, J.~Haas, Achieving the orthoplex bound and constructing weighted
  complex projective 2-designs with singer sets, arXiv preprint
  arXiv:1509.05333.

\bibitem{MR3429342}
H.~Z{\"o}rlein, M.~Bossert,
  \href{http://dx.doi.org/10.1109/TSP.2015.2477052}{Coherence optimization and
  best complex antipodal spherical codes}, IEEE Trans. Signal Process. 63~(24)
  (2015) 6606--6615.
\newblock \href {http://dx.doi.org/10.1109/TSP.2015.2477052}
  {\path{doi:10.1109/TSP.2015.2477052}}.
\newline\urlprefix\url{http://dx.doi.org/10.1109/TSP.2015.2477052}

\bibitem{MR2711357}
O.~Oktay,
  \href{http://gateway.proquest.com.ezproxy.lib.uh.edu/openurl?url_ver=Z39.88-2004&rft_val_fmt=info:ofi/fmt:kev:mtx:dissertation&res_dat=xri:pqdiss&rft_dat=xri:pqdiss:3297348}{Frame
  quantization theory and equiangular tight frames}, ProQuest LLC, Ann Arbor,
  MI, 2007, dissertation (Ph.D.)--University of Maryland, College Park, MD.
\newline\urlprefix\url{http://gateway.proquest.com.ezproxy.lib.uh.edu/openurl?url_ver=Z39.88-2004&rft_val_fmt=info:ofi/fmt:kev:mtx:dissertation&res_dat=xri:pqdiss&rft_dat=xri:pqdiss:3297348}

\bibitem{MR1418961}
J.~H. Conway, R.~H. Hardin, N.~J.~A. Sloane,
  \href{http://projecteuclid.org.ezproxy.lib.uh.edu/euclid.em/1047565645}{Packing
  lines, planes, etc.: packings in {G}rassmannian spaces}, Experiment. Math.
  5~(2) (1996) 139--159.
\newline\urlprefix\url{http://projecteuclid.org.ezproxy.lib.uh.edu/euclid.em/1047565645}

\bibitem{MR0074013}
R.~A. Rankin, The closest packing of spherical caps in {$n$} dimensions, Proc.
  Glasgow Math. Assoc. 2 (1955) 139--144.

\bibitem{MR2882247}
R.-A. Pitaval, H.-L. M\"a\"att\"anen, K.~Schober, O.~Tirkkonen, R.~Wichman,
  \href{http://dx.doi.org/10.1109/TIT.2011.2165820}{Beamforming codebooks for
  two transmit antenna systems based on optimum {G}rassmannian packings}, IEEE
  Trans. Inform. Theory 57~(10) (2011) 6591--6602.
\newblock \href {http://dx.doi.org/10.1109/TIT.2011.2165820}
  {\path{doi:10.1109/TIT.2011.2165820}}.
\newline\urlprefix\url{http://dx.doi.org/10.1109/TIT.2011.2165820}

\bibitem{fejes1943covering}
L.~Fejes-T{\'o}th, On covering a spherical surface with equal spherical caps,
  Matematikai {\'e}s fizikai lapok 50 (1943) 40--46.

\bibitem{BK06}
J.~J. Benedetto, J.~D. Kolesar, Geometric properties of {G}rassmannian frames
  for ${R}^2$ and ${R}^3$, EURASIP J. Appl. Signal Process. 2006 (2006) 1--17.

\bibitem{welch:bound}
L.~R. Welch, Lower bounds on the maximum cross correlation of signals, IEEE
  Trans. on Information Theory 20~(3) (1974) 397--9.

\bibitem{henkel05}
O.~Henkel, Sphere packing bounds in the grassmann and stiefel manifolds, IEEE
  Trans. Inf. Theory 51~(10) (2005) 3445 3456.

\bibitem{MR2337670}
A.~Roy, A.~J. Scott, \href{http://dx.doi.org/10.1063/1.2748617}{Weighted
  complex projective 2-designs from bases: optimal state determination by
  orthogonal measurements}, J. Math. Phys. 48~(7) (2007) 072110, 24.
\newblock \href {http://dx.doi.org/10.1063/1.2748617}
  {\path{doi:10.1063/1.2748617}}.
\newline\urlprefix\url{http://dx.doi.org/10.1063/1.2748617}

\bibitem{MR2269701}
A.~J. Scott, \href{http://dx.doi.org/10.1088/0305-4470/39/43/009}{Tight
  informationally complete quantum measurements}, J. Phys. A 39~(43) (2006)
  13507--13530.
\newblock \href {http://dx.doi.org/10.1088/0305-4470/39/43/009}
  {\path{doi:10.1088/0305-4470/39/43/009}}.
\newline\urlprefix\url{http://dx.doi.org/10.1088/0305-4470/39/43/009}

\bibitem{MR3325226}
A.~Barg, A.~Glazyrin, K.~A. Okoudjou, W.-H. Yu,
  \href{http://dx.doi.org/10.1016/j.laa.2015.02.020}{Finite two-distance tight
  frames}, Linear Algebra Appl. 475 (2015) 163--175.
\newblock \href {http://dx.doi.org/10.1016/j.laa.2015.02.020}
  {\path{doi:10.1016/j.laa.2015.02.020}}.
\newline\urlprefix\url{http://dx.doi.org/10.1016/j.laa.2015.02.020}

\end{thebibliography}

\end{document}